\documentclass[a4paper,oneside,11pt,reqno]{amsart}

\usepackage[utf8]{inputenc}
\usepackage{stmaryrd}
\usepackage{bbm}
\usepackage[dvipsnames]{xcolor}
\usepackage{xypic}
\usepackage{mathrsfs}
\usepackage{enumerate}
\usepackage{centernot}
\usepackage{latexsym,amsxtra}
\usepackage{dsfont}
\usepackage{slashed}
\usepackage[all]{xy}
\usepackage{amscd,graphics}
\usepackage{amsmath,amsfonts,amsthm,amssymb}
\usepackage{latexsym,amsmath}
\usepackage{fancyhdr}

\usepackage{libertine}
\usepackage{newtxmath}

\usepackage{cleveref}
\usepackage{comment} 
\usepackage{braket}
\usepackage{caption}
\usepackage{enumitem}

\theoremstyle{plain}
\newtheorem{thm}{Theorem}[section]
\crefname{thm}{Theorem}{Theorems}
\Crefname{thm}{Theorem}{Theorems}
\newtheorem{pro}[thm]{Proposition}
\crefname{pro}{Proposition}{Propositions}
\Crefname{pro}{Proposition}{Propositions}
\newtheorem{lem}[thm]{Lemma}
\crefname{lem}{Lemma}{Lemmas}
\Crefname{lem}{Lemma}{Lemmas}
\newtheorem{cor}[thm]{Corollary}
\crefname{cor}{Corollary}{Corollaries}
\Crefname{cor}{Corollary}{Corollaries}

\crefname{conj}{Conjecture}{Conjectures}
\Crefname{conj}{Conjecture}{Conjectures}

\crefname{cons}{Construction}{Constructions}
\Crefname{cons}{Construction}{Constructions}

\crefname{claim}{Claim}{Claims}
\Crefname{claim}{Claim}{Claims}

\crefname{property}{Property}{Properties}
\Crefname{property}{Property}{Properties}

\crefname{problem}{Problem}{Problems}
\Crefname{problem}{Problem}{Problems}

\theoremstyle{definition}

\crefname{defi}{Definition}{Definitions}
\Crefname{defi}{Definition}{Definitions}

\crefname{nota}{Notation}{Notations}
\Crefname{nota}{Notation}{Notations}

\crefname{convention}{Convention}{Conventions}
\Crefname{convention}{Convention}{Conventions}

\crefname{cond}{Condition}{Conditions}
\Crefname{cond}{Condition}{Conditions}
\newtheorem{assum}[thm]{Assumption}
\crefname{assum}{Assumption}{Assumptions}
\Crefname{assum}{Assumption}{Assumptions}

\theoremstyle{remark}
\newtheorem{rmk}[thm]{Remark}
\crefname{rmk}{Remark}{Remarks}
\Crefname{rmk}{Remark}{Remarks}

\crefname{ex}{Example}{Examples}
\Crefname{ex}{Example}{Examples}

\newtheorem{ques}[thm]{Question}
\crefname{ques}{Question}{Questions}
\Crefname{ques}{Question}{Questions}

\crefname{section}{Section}{Sections}
\Crefname{section}{Section}{Sections}
\crefname{subsection}{Subsection}{Subsections}
\Crefname{subsection}{Subsection}{Subsections}
\crefname{figure}{Figure}{Figures}
\Crefname{figure}{Figure}{Figures}

\newtheorem*{acknowledgement}{Acknowledgement}

\newcommand{\BDiff}{B\textnormal{Diff}}
\newcommand{\TDiff}{\textnormal{TDiff}}

\newcommand{\Homeo}{\textnormal{Homeo}}
\newcommand{\BHomeo}{B\textnormal{Homeo}}

\newcommand{\SWcal}{\mathcal{SW}}

\newcommand{\SWbbhalftot}{\mathbb{SW}_{\mathrm{half\mathchar`- tot}}}
\newcommand{\SWcalhalftot}{\mathcal{SW}_{\mathrm{half\mathchar`- tot}}}

\newcommand{\id}{\textnormal{id}}
\newcommand{\Z}{\mathbb{Z}}
\newcommand{\N}{\mathbb{N}}
\newcommand{\R}{\mathbb{R}}

\newcommand{\CP}{\mathbb{CP}}
\newcommand{\fraks}{\mathfrak{s}}
\newcommand{\frakt}{\mathfrak{t}}

\newcommand{\scrR}{\mathscr{R}}

\newcommand{\calS}{\mathcal{S}}
\newcommand{\calM}{\mathcal{M}}

\newcommand{\Conj}{\mathrm{Conj}}
\DeclareMathOperator{\Aut}{Aut}

\DeclareMathOperator{\Map}{Map}

\newcommand{\Spinc}{\mathrm{Spin}^{c}}

\newcommand{\circPi}{\mathring{\Pi}}
\newcommand{\del}{\partial}

\newcommand{\bbS}{\mathbb{S}}
\newcommand{\divi}{\mathrm{div}}
\newcommand{\Bhalf}{\mathcal{B}_{\mathrm{half}}}

\title{The homology of moduli spaces of 4-manifolds may be infinitely generated}

\author{Hokuto Konno}

\date{}

\begin{document}

\maketitle

\begin{abstract}
For a simply-connected closed manifold $X$ of $\dim X \neq 4$, the mapping class group $\pi_0(\mathrm{Diff}(X))$ is known to be finitely generated. We prove that analogous finite generation fails in dimension 4. Namely, we show that there exist simply-connected closed smooth $4$-manifolds whose mapping class groups are not finitely generated. More generally, for each $k>0$, we prove that there are simply-connected closed smooth $4$-manifolds $X$ for which $H_k(B\mathrm{Diff}(X);\mathbb{Z})$ are not finitely generated. The infinitely generated subgroup of $H_k(B\mathrm{Diff}(X);\mathbb{Z})$ which we detect are topologically trivial, and unstable under the connected sum of $S^2 \times S^2$. These results are proven by constructing and computing an infinite family of characteristic classes using Seiberg--Witten theory.
\end{abstract}

\section{Introduction}
\label{section Main results}

\subsection{Main results}
\label{subsec main results}
The purpose of this paper is to present a new special phenomenon in dimension $4$ in terms diffeomorphism groups.
To describe our result, let $\mathrm{Diff}(X)$ denote the diffeomorphism group equipped with the $C^\infty$-topology for a given a smooth manifold $X$.
It is known that the mapping class group $\pi_0(\mathrm{Diff}(X))$ is finitely generated, if $X$ is simply-connected, closed, and $\dim X \neq 4$.
For $\dim X \geq 5$, this is due to Sullivan~\cite[Theorem~(13.3)]{Sull77}.
For $\dim X \leq 3$, finite generation holds even dropping the simple-connectivity:
In fact, even stronger finiteness is known in all dimensions $\neq 4$ (see \cref{subsection Finiteness dim not 4}, including a remark for $\dim=5$).

We prove that analogous finite generation fails in dimension $4$.
Namely, we show that there exist simply-connected closed smooth $4$-manifolds whose mapping class groups are infinitely generated:

\begin{thm}
\label{thm: main pi0}
For $n \geq 2$, set $X=E(n)\#S^2 \times S^2$.
Then $\pi_0(\mathrm{Diff}(X))$ is not finitely generated.   
\end{thm}

Here $E(n)$ denotes the simply-connected elliptic surface of degree $n$ without multiple fiber.
As is well-known, $E(n)\#S^2 \times S^2$ can be written in terms of further basic $4$-manifolds
(e.g. \cite[Corollary~8]{Gom91}):
\[
E(n)\#S^2 \times S^2
\cong 
\left\{
\begin{array}{ll}
2n\CP^2\#10n\overline{\CP}^2 & \text{ for $n$ odd},\\
m(K3\#S^2 \times S^2) & \text{ for $n =2m$ even}.
\end{array}
\right.
\]

\begin{rmk}
After completing a preprint version of this paper, the author was informed that David Baraglia \cite{Baraglia23mapping} also proved that the mapping class groups of simply-connected $4$-manifolds can be infinitely generated.
Baraglia's proof is based on essentially the same method as ours, however we obtained our proofs completely independently.
\end{rmk}

\begin{rmk}[Topological mapping class group]
\label{rem: top MCG}
Let $\Homeo(X)$ denote the homeomorphism group of $X$.
If $X$ is a simply-connected closed topological $4$-manifold, then $\pi_0(\Homeo(X))$ is finitely generated.
This follows from a result by
Quinn~\cite{Q86} and Perron~\cite{P86}.
Thus infinite generation exhibited in \cref{thm: main pi0} is special to the $4$-dimensional smooth category.
\end{rmk}

\cref{thm: main pi0} is a consequence of a more general result on the (co)homology of the moduli spaces $B\mathrm{Diff}(X)$ of $4$-manifolds $X$.
The (co)homology of $B\mathrm{Diff}(X)$ is a fundamental object, since it corresponds to the set of characteristic classes of fiber bundles with fiber $X$.
We shall prove that, for each $k\geq0$, there exist simply-connected closed smooth $4$-manifolds $X$ where $H_{k}(B\mathrm{Diff}(X);\Z)$ are infinitely generated.
More strongly, we shall see that the ``topologically trivial parts" of $H_{k}(B\mathrm{Diff}(X);\Z)$ can be infinitely generated.
To state this, let $i : \mathrm{Diff}(X) \hookrightarrow \Homeo(X)$ denote the inclusion map into the homeomorphism group.
We shall prove:

\begin{thm}
\label{thm: main general}
For $n \geq 2$ and $k \geq 1$, 
set $X = E(n)\# k S^2 \times S^2$.
Then
\[
\ker(i_\ast : H_k(B\mathrm{Diff}(X);\Z) \to H_k(B\Homeo(X);\Z))
\]
contains a direct summand isomorphic to $(\Z/2)^\infty$.
In particular, $H_{k}(B\mathrm{Diff}(X);\Z)$ is not finitely generated.
\end{thm}

Here $(\Z/2)^\infty$ denotes the countably infinite direct sum $\bigoplus_{\N}\Z/2$.
Rephrasing \cref{thm: main general} for $k=1$, we have the following result, which immediately implies \cref{thm: main pi0}:

\begin{cor}
\label{cor abelianization}
For $n \geq 2$, set $X=E(n)\#S^2\times S^2$.
Then
\[
\ker(i_\ast : \pi_0(\mathrm{Diff}(X))_{\mathrm{ab}} \to \pi_0(\Homeo(X))_{\mathrm{ab}})
\]
contains a direct summand isomorphic to $(\Z/2)^\infty$.
Here the subscript $\mathrm{ab}$ indicates the abelianization.
\end{cor}

To our knowledge, \cref{thm: main general} gives the first examples of simply-connected closed manifolds $X$ where $H_k(B\mathrm{Diff}(X);\Z)$ are confirmed to be infinitely generated for given $k\geq1$ (for $k=1$, this follows also from the aforementioned result by Baraglia~\cite{Baraglia23mapping}).
It is worth noting that there are several established and expected finiteness in dim $\neq 4$ (\cref{rem: finiteness in other dim}).
Thus infiniteness given in \cref{thm: main general} reflects a specialty of dimension 4, described in terms of characteristic classes of fiber bundles (see \cref{rmk cohomology} below).

\begin{rmk}
\label{rmk cohomology}
In terms of cohomology,
\cref{thm: main general}  deduces that non-topological characteristic classes may form a group isomorphic to $(\Z/2)^{\infty}$ for some $4$-manifolds.
Here we call an element of 
\begin{align*}
\mathrm{coker}(i^\ast : H^k(B\Homeo(X);\Z/2) \to H^k(B\mathrm{Diff}(X);\Z/2))
\end{align*}
a {\it non-topological characteristic class} (over $\Z/2$).
This $(\Z/2)^{\infty}$-subgroup is generated by gauge-theoretic characteristic classes we shall introduce (\cref{subsec Scheme of the proof}).
In contrast, the Mumford--Morita--Miller classes, the most basic characteristic class of manifold bundles, are topological over a field of characteristic 2 or 0 \cite{EbertRandal-Williams14}.
\end{rmk}

It is worth noting a consequence about stabilization.
Recently, Lin and the author \cite{konno2022homological} proved that the moduli spaces $B\mathrm{Diff}(X)$ of $4$-manifolds $X$ do not satisfy homological stability with respect to connected sums of $S^2 \times S^2$, unlike what happens in dimension $\neq 4$ \cite{Harer85,Galatius2018}.
The proof of \cref{thm: main general} shows also that the unstable part of $H_\ast(B\mathrm{Diff}(X))$ may be infinitely generated.
To state this precisely, given a closed $4$-manifold $X$, take a smoothly embedded $4$-disk $D^4$ in $X$, and set $\mathring{X} = X \setminus \mathrm{Int}(D^4)$.
Let $\mathrm{Diff}_\del(\mathring{X})$ denote the group of diffeomorphisms that are the identity near $\del \mathring{X}$.
Form the (inner) connected sum $\mathring{X}\#S^2 \times S^2$ by $\mathring{X}\cup_{S^3} ((S^3 \times [0,1])\#S^2 \times S^2)$.
Then one can define the stabilization map
\[
s : \mathrm{Diff}_\del(\mathring{X})
\to \mathrm{Diff}_\del(\mathring{X}\#S^2 \times S^2)
\]
by extending by the identity on $(S^3 \times [0,1])\#S^2 \times S^2$.
We shall prove:

\begin{thm}
\label{thm: main stabilization}
For $n \geq 2$ and $k \geq 1$, 
set $X = E(n)\# k S^2 \times S^2$.
Then the kernel of the induced map 
\[
s_\ast : H_k(B\mathrm{Diff}_\del(\mathring{X});\Z) \to H_k(B\mathrm{Diff}_\del(\mathring{X}\# S^2 \times S^2);\Z)
\]
contains a direct summand isomorphic to $(\Z/2)^\infty$.
\end{thm}

\subsection{Related results}
The following remarks list related results in  more detail:

\begin{rmk}[Other finiteness in $\dim \neq 4$]
\label{rem: finiteness in other dim}
Let us compare
infinite generation of $H_k(\BDiff(X);\Z)$ in
\cref{thm: main general} with other dimensions.
For a manifold $X$ of even $\dim \geq 6$ and with finite $\pi_1(X)$, Bustamante--Krannich--Kupers proved that $H_k(B\mathrm{Diff}(X);\Z)$ is finitely generated for each $k$ \cite[Corollary~B]{bustamante2023finiteness}.
Also, in his earlier paper, Kupers~\cite[Corollary C]{Kupers19} has proved an analogous statement for a 2-connected manifold $X$ of $\dim \neq 4,5,7$.
As mentioned in \cite{bustamante2023finiteness}, there is an expectation that finiteness may hold even dropping the 2-connectivity.
For finiteness of mapping class groups in dimension $\neq 4$, see \cref{subsection Finiteness dim not 4}.
\end{rmk}

\begin{rmk}[Infiniteness of the Torelli group]
Given a smooth closed oriented $4$-manifold $X$, let $\TDiff(X)$ denote the {\it Torelli diffeomorphism group}, i.e. the group of diffeomorphisms acting trivially on $H_\ast(X;\Z)$.
Ruberman~\cite[Theorem A]{Rub99} proved that $\pi_0(\TDiff(X))$ are infinitely generated for $X=E(n)\#\CP^2\#k\overline{\CP}^2$ with $k \geq 2$.
Note that infinite generation of $\pi_0(\mathrm{Diff}(X))$ does not necessarily follow from that of $\pi_0(\TDiff(X))$.
For example, $\pi_0(\TDiff(X))$ is infinitely generated if one takes $X$ to be the genus 2 surface, whereas $\pi_0(\mathrm{Diff}(X))$ is finitely generated \cite{McCulloughgenus2}.
This phenomenon may occur since the index of $\pi_0(\TDiff(X))$ in $\pi_0(\mathrm{Diff}(X))$ is infinite (also for $X$ in \cref{thm: main general}), and an infinite index subgroup of a finitely generated group is not necessarily finitely generated.
\end{rmk}

\begin{rmk}[Other infiniteness in $\dim=4$]
Baraglia~\cite{B21} and Lin~\cite{lin2022family} proved that $\pi_1(\mathrm{Diff}(X))$ have infinite-rank summands for some simply-connected (irreducible) $4$-manifolds $X$.
Further, Auckly--Ruberman~\cite{AucklyRubermanonSW} announced that, for each $k>0$, there are simply-connected $4$-manifolds $X$ such that $\pi_k(\mathrm{Diff}(X))$ have infinite-rank summands.
They prove an analogous result also for $H_{k}(B\TDiff(X);\Z)$.
\end{rmk}

\begin{rmk}[Non-simply-connected manifolds]
For non-simply-connected manifolds of $\dim \geq 4$, it has been known that the mapping class group may be infinitely generated.
For instance, Hatcher~\cite[Theorem~4.1]{Hat76} proved that the mapping class groups of the tori $T^n$ for $n \geq 5$ are infinitely generated.
In dimension 4, Budney--Gabai~\cite{budney2021knotted} and Watanabe~\cite{watanabe2023thetagraph} gave examples of non-simply-connected $4$-manifolds whose mapping class groups are infinitely generated.
Budney--Gabai~\cite{budney2023automorphism} also proved that their infinitely generated subgroups of mapping class groups are non-trivial also in the topological category. 
\end{rmk}

\subsection{Scheme of the proof}
\label{subsec Scheme of the proof}

Now we describe the idea of proofs of our results given in \cref{subsec main results}.
We shall introduce an infinite family of characteristic classes
\begin{align}
\label{eq: intro SW halftot}
\SWbbhalftot^k(X,\calS)
\in H^k(B\mathrm{Diff}^+(X);\Z/2)
\end{align}
using Seiberg--Witten theory for families.
Here $\calS$ are $\mathrm{Diff}^+(X)$-invariant subsets of the set of spin$^c$ structures $\Spinc(X,k)$ on $X$ with Seiberg--Witten formal dimension $-k$, divided by the charge conjugation (see \cref{subsec: output of char class} for the precise definition).

A characteristic class for families of $4$-manifolds using Seiberg--Witten theory was introduced by the author \cite{K21}, under the assumption that the monodromies of families preserve a given spin$^c$ structure.
Later, Lin and the author \cite{konno2022homological} defined a version without the assumption on monodromy.
The classes \eqref{eq: intro SW halftot} are refinements of the characteristic class defined in \cite{konno2022homological}.

Using the characteristic classes \eqref{eq: intro SW halftot},
we can define a homomorphism
\begin{align*}
\bigoplus_{\calS}
\langle \SWbbhalftot^k(X,\calS), - \rangle :
H_k(B\mathrm{Diff}^+(X);\Z) \to \bigoplus_{\calS} \Z/2.
\end{align*}
The above results follow by seeing that this homomorphism has infinitely generated image in $\bigoplus_{\calS} \Z/2$ for some class of $4$-manifolds $X$, including $X=E(n)\#kS^2\times S^2$.
More precisely, we shall see that $\langle\SWbbhalftot^k(X,\calS), -\rangle$ are non-trivial for infinitely many orbits $\calS$ for the action of $\mathrm{Diff}^+(X)$ on $\Spinc(X,k)$, which are distinguished by divisibilities of the first Chern classes.

\subsection{Structure of the paper}
The following is an outline of the sections of the paper.
In \cref{sec: Families over a torus}, we construct infinitely many homology classes of $B\mathrm{Diff}(X)$ for some class of $4$-manifolds $X$, which will be shown to be linearly independent over $\Z/2$.
The most general statement is given as \cref{thm: generating homoogy}, which implies all results explained above.
In \cref{section Characteristic classes from Seiberg--Witten theory}, we construct characteristic classes \eqref{eq: intro SW halftot}, and compute them in \cref{section Computing the invariant} to prove \cref{thm: generating homoogy}.

\begin{acknowledgement}
The author is grateful to Dave Auckly, Inanc Baykur, and Danny Ruberman for enlightening conversations.
Especially, the author wishes to thank Danny Ruberman for several helpful discussions during this project.
The author would like to thank Sander Kupers, Jianfeng Lin, Mike Miller Eismeier, Masaki Taniguchi, and Tadayuki Watanabe for their helpful comments on a draft of the paper.
The author is grateful to David Baraglia for informing the author of his paper \cite{Baraglia23mapping}, as well as comments on a draft of the current paper.
The author is grateful to an anonymous referee for giving several constructive comments, which greatly improved this paper.
This project was prompted at SwissMAP Research Station in Les Diablerets, the conference ``Mapping class groups: pronilpotent and cohomological approaches", and the author wishes to thank the organizers for giving him an opportunity to join the conference.
Lastly, the author wants to acknowledge the hospitality of MIT where this project was conducted.
The author was partially supported by JSPS KAKENHI Grant Number 21K13785 and JSPS Overseas Research Fellowship.
\end{acknowledgement}

\section{Construction of homology classes}
\label{sec: Families over a torus}

In this \lcnamecref{sec: Families over a torus}, we construct infinitely many homology classes of $B\mathrm{Diff}(X)$ for $4$-manifolds $X$ with certain conditions, which will be shown to be linearly independent over $\Z/2$.

\subsection{Mod 2 basic classes in $H^2(M;\Z)/\Aut(H^2(M;\Z))$}

A building block of the construction of homology classes of $B\mathrm{Diff}^+(X)$ is a $4$-manifold $M$ that admits infinitely many exotic structures.
This is inspired by Ruberman's argument \cite{Rub99} in his work on Torelli groups.
(See also Auckly's recent work \cite{auckly2023smoothly} for one version of Ruberman's argument in a Seiberg--Witten context.)
Compared with \cite{Rub99,auckly2023smoothly},
what we newly need to require is that those exotic structures are distinguished by mod 2 basic classes that are distinct in 
\[
H^2(M;\Z)/\Aut(H^2(M;\Z)),
\]
the quotient of $H^2(M;\Z)$ by the automorphism group of the intersection form.
This is a reflection that we shall consider the whole diffeomorphism group, rather than the Torelli diffeomorphism group.

To formulate this precisely, let us introduce some notation.
Let $M$ be a smooth simply-connected closed oriented $4$-manifold with $b^+(M) \geq 2$.
Since $H^2(M;\Z)$ has no torsion, we can identify a spin$^c$ structure on $M$ with a characteristic element in $H^2(M;\Z)$.
Recall that a characteristic element $c \in H^2(M;\Z)$ is called a basic class if $SW(M,c)$, the Seiberg--Witten invariant, is non-zero.
If $SW(M,c) \neq 0$ mod $2$, we say that $c$ is a {\it mod 2 basic class}.
For simplicity, whenever we say that $c$ is a (mod 2) basic class, we further impose that the formal dimension of $c$ is zero (see \eqref{eq: formal dim}).
We denote by $\mathcal{B}_2(M)$ the set of mod 2 basic classes of $M$.
Note that $\mathcal{B}_2(M)$ is preserved the $\Z/2$-action on $H^2(M;\Z)$ via multiplication by $-1$.
For a non-zero cohomology class $x \in H^2(M;\Z)$, let $\divi(x)$ denote the divisibility of $x$, namely 
\[
\divi(x) = \max\left\{n \in \Z \mid \exists y \in H^2(M;\Z) \text{ such that } ny=x\right\}.
\]
For the zero element, we formally set $\divi(0)=0$ in this paper.
For a characteristic element $c \in H^2(M;\Z)$, define
\[
N(M;c)
= \#\{[x] \in \mathcal{B}_2(M)/(\Z/2) \mid \divi(x)=\divi(c),\  x^2=c^2\},
\]
where $x^2$ denotes the self-intersection of $x$.
In this \lcnamecref{sec: Families over a torus},
we consider a $4$-manifold $M$ satisfy the following assumption:

\begin{assum}
\label{assumption on M}
Let $M$ be an indefinite smooth simply-connected closed oriented $4$-manifold with $b^+(M) \geq 2$. 
Assume that there exist smooth $4$-manifolds $\{M_i\}_{i=1}^\infty$ that satisfies the following three properties:

\begin{itemize}
\item [(i)] Each $M_i$ is homeomorphic to $M$.  
\item [(ii)] For every $i$, $M_i\#S^2\times S^2$ is diffeomorphic to $M\#S^2\times S^2$.
\item [(iii)] For every $i$, there exists a mod 2 basic class $c_i$ on $M_i$ with $N(M_i;c_i)$ odd, and the sequence $\{c_i\}_{i=1}^\infty$ satisfies that  $\divi(c_i) \to +\infty$ as $i \to +\infty$.
\end{itemize}    
\end{assum}

It is worth adding notes on the last property (iii) of \cref{assumption on M}.
The principal intention of (iii) is to ensure that the mod 2 basic classes are distinct even in the quotient $H^2(M;\Z)/\Aut(H^2(M;\Z))$ (after passing to a subsequence, if necessary).
For most $4$-manifolds, increasing either divisibilities or self-intersections is the only possible way to get infinitely many characteristics distinct in
\[
H^2(M;\Z)/\Aut(H^2(M;\Z))
\]
({\it cf.} \cref{lem: Wall}).
The reason why we suppose $N(M_i; c_i)$ is odd is that we want to control sums of mod 2 Seiberg--Witten invariants over some class of spin$^c$ structures.

As a series of examples of $M$ satisfying \cref{assumption on M}, we have:

\begin{lem}
\label{lem En satisfies Assumption}
For $n\geq1$, $M=E(n)$ satisfies \cref{assumption on M}.
\end{lem}

To see \cref{lem En satisfies Assumption}, let us consider logarithmic transformations.
For $n \geq 2$ and $i \geq 1$,
let $E(n;i)$ denote the logarithmic transformation of order $i>0$ performed on $E(n)$.
i.e. $E(n;i)$ is the elliptic surface of degree $n$ with a single multiple fiber of order $i$.
(Note that $E(n;1)=E(n)$.)
The Seiberg--Witten invariants of $E(n;i)$ were computed by Fintushel--Stern~\cite{FS97}.
For readers' convenience,
we recall the result following their survey \cite[Lecture~2]{FS6lec}.

In general, let $Z$ be an oriented closed smooth $4$-manifold with $b^+(Z) \geq 2$, without torsion in $H^2(Z;\Z)$.
Consider the Laurent polynomial
\[
\SWcal_Z = \sum_{c}SW(Z,c)t_c.
\]
Here $c \in H^2(Z;\Z)$ are characteristic elements and $t_c$ are formal variables in $\Z[H^2(Z;\Z)]$ corresponding to $c$.
Note that $t_{c}t_{c'}=t_{c+c'}$, in particular $t_c^m=t_{mc}$ for $m \in \Z$.

Now consider $Z=E(n;i)$.
Let $F \in H_2(E(n;i);\Z)$ be the class that represents a generic fiber of the elliptic fibration.
The multiple fiber of $E(n;i)$ represents a primitive homology class, which is given by $F/i$.
Let $F_i$ denote the Poincar\'{e} dual of $F/i$ and set $t = t_{F_i}$.
Then the Seiberg--Witten polynomial for $E(n;i)$ is given by
\begin{align}
\SWcal_{E(n;i)} = (t^{i}-t^{-i})^{n-2}(t^{i-1} + t^{i-3} + \dots + t^{1-i}).
\label{SW inv of elliptic surfaces}
\end{align}

\begin{lem}
\label{lem: SW inv for elliptic surfaces}
The classes $\pm(ni-i-1)F_i \in H^2(E(n;r);\Z)$ are mod 2 basic classes of $E(n;i)$.
Further, we have $\divi((ni-i-1)F_i)=ni-i-1$, and there is no mod 2 basic class of $\divi=ni-i-1$ other than $\pm(ni-i-1)F_i$.
\end{lem}

\begin{proof}
Since the right-hand side of \eqref{SW inv of elliptic surfaces} is a polynomial only in $t$, all basic classes of $E(n;i)$ lie in the set $\{kF_i \in H^2(E(n;i);\Z) \mid k \in \Z\}$.
Thus, for each $d \geq 1$, we have at most two basic classes of $\divi=d$, related by multiplication 
by $-1$ if exist.
On the other hand, the top degree term of the right-hand side of \eqref{SW inv of elliptic surfaces} is given by $t^{(n-2)i}t^{i-1}=t_{(ni-i-1)F_i}$.
Thus $\pm(ni-i-1)F_i$ are mod 2 basic classes.
The assertion of the \lcnamecref{lem: SW inv for elliptic surfaces} follows from this by recalling that $F_i$ is a primitive class.
\end{proof}

\begin{proof}[Proof of \cref{lem En satisfies Assumption}]
Set $M_i=E(n;i)$.
Here $i$ runs over the natural numbers, but we restrict $i$ to be odd if $n$ is spin, so that the spinness of $M_{i}$ is the same as that of $M$.
Then $M_i$ satisfies the properties (i) and (ii) of \cref{assumption on M} by \cite{Gom91}.
To check the property (iii), set $c_i=(ni-i-1)F_i$.
Then it follows from \cref{lem: SW inv for elliptic surfaces} that $\divi(c_i) \to +\infty$ as $i \to +\infty$, and $N(M_i;c_i)=1$ for all $i \geq 1$.
Hence the property (iii) is satisfied.
This completes the proof.
\end{proof}

\subsection{Families over the torus}
\label{subsectionFamilies over the torus}

Fix $k>0$, and let us take a $4$-manifold $M$ satisfying  \cref{assumption on M}.
Fix a diffeomorphism $M_i\# S^2 \times S^2 \to M\# S^2 \times S^2$ and identify $M_i\# kS^2 \times S^2$ with $X$ for every $i$.
Set $X = M\#kS^2 \times S^2$.

We recall a construction of a smooth fiber bundle over $T^k$
with fiber $X$ considered in \cite{K21,konno2022homological}.
Define an orientation-preserving diffeomorphism $f_0 : S^2 \times S^2 \to S^2 \times S^2$ by $f_0(x,y) = (r(x), r(y))$, where $r : S^2 \to S^2$ is the reflection about the equator.
By isotoping $f_0$, we can obtain a diffeomorphism $f : S^2 \times S^2 \to S^2 \times S^2$ that fixes a disk $D^4 \subset S^2 \times S^2$ pointwise.
Take copies $f_1, \ldots, f_k$ of $f$, and implant them into $M_i\# kS^2\times S^2$ for each $i$, by extending by the identity.
Thus we obtain diffeomorphisms $f_1, \ldots, f_k : M_i\# kS^2\times S^2 \to M_i\# kS^2\times S^2$.
Since the supports of $f_1, \ldots, f_k$ are mutually disjoint, and $f_1, \ldots, f_k$ commute each other.
Using these commuting diffeomorphisms, we can form the multiple mapping torus $E_i \to T^k$, which is a smooth fiber bundle with fiber $M_i\# kS^2\times S^2$.
Using the fixed identification between $M_i\# kS^2\times S^2$ and $X$, we obtain smooth fiber bundles (denoted by the same notation) $X \to E_i \to T^k$ with fiber $X$.
Since $f$ is orientation-preserving, the resulting fiber bundles $E_i$ are oriented, i.e. the structure group reduces to $\mathrm{Diff}^+(X)$, the orientation-preserving diffeomorphism group.

For each $i\geq1$, regard $E_i$ as a continuous map $E_i : T^k \to B\mathrm{Diff}^+(X)$.
Now we set 
\begin{align}
\label{eq: alpha}
\alpha_i
:= (E_1)_\ast([T^k]) - (E_i)_\ast([T^k])
\in H_k(B\mathrm{Diff}^+(X);\Z).
\end{align}
This construction of the homology class $\alpha_i$ is the same as the one in \cite[Proof of Theorem 3.10]{konno2022homological}, except only for a condition on the Seiberg--Witten invariants of $4$-manifolds.
The origin of this construction is the first examples of exotic diffeomorphisms of $4$-manifolds due to Ruberman \cite{Rub98}.

Let us observe a few properties of $\alpha_i$:

\begin{lem}
\label{lem: topologically trivial}
The homology class $\alpha_i$ lies in 
\[
\ker(i_\ast : H_k(B\mathrm{Diff}^+(X);\Z) \to H_k(B\Homeo^+(X);\Z)).
\]
\end{lem}

\begin{proof}
Since $M_1$ and $M_i$ are homeomorphic, $E_1$ and $E_i$ are isomorphic as topological bundles.
The assertion follows from this. 
\end{proof}

For each $i$, fix a smoothly embedded $4$-disk $D^4_i \subset M_i$ and similarly take $D^4 \subset M$.
Set $\mathring{M}_i = M_i \setminus \mathrm{Int}(D^4_i)$ and $\mathring{M} = M \setminus \mathrm{Int}(D^4)$.
We can find a diffeomorphism $\psi_i : M_i\#S^2\times S^2 \to M\#S^2\times S^2$ with  $\psi_i(D^4_i)=D^4$ that respect parametrizations of $D^4_i$ and $D^4$.
Thus we can identify $\mathring{M}_i\#S^2 \times S^2$ with  $\mathring{M}\#S^2 \times S^2$.
Using these identifications, the construction of $E_i$ can be carried out with fixing the $4$-disks, so we also have a homology class $\mathring{\alpha}_i \in H_k(B\mathrm{Diff}_\del(\mathring{X});\Z)$ defined similarly to $\alpha_i$, namely
\begin{align*}
\mathring{\alpha}_i
:= (E_1)_\ast([T^k]) - (E_i)_\ast([T^k])
\in H_k(B\mathrm{Diff}_\del(\mathring{X});\Z),
\end{align*}
where $E_i$ are regarded as maps $E_i : T^k \to \BDiff_\del(\mathring{X})$.
Letting $\rho : \mathrm{Diff}_\del(\mathring{X}) \to \mathrm{Diff}^+(X)$ be the extension map by the identity, we have that $\alpha_i$ is the image of $\mathring{\alpha}_i$ under the induced map
\[
\rho_\ast : H_k(B\mathrm{Diff}_\del(\mathring{X});\Z) \to H_k(B\mathrm{Diff}^+(X)\Z).
\]

\begin{lem}
\label{lem: stabilization}
The homology class $\mathring{\alpha}_i$ lies in the kernel of
\[
s_\ast : H_k(B\mathrm{Diff}_\del(\mathring{X});\Z) \to H_k(B\mathrm{Diff}_\del(\mathring{X}\# S^2 \times S^2);\Z).
\]
\end{lem}

\begin{proof}
Let $E_S \to T^k$ denote the trivialized bundle with fiber $(S^2 \times S^2) \setminus \mathrm{Int}(D^4)$.
Since $M_1 \# S^2 \times S^2$ and $M_i\# S^2 \times S^2$ are diffeomorphic, the stabilized bundles $E_1\#_{\mathrm{fib}}E_S$ and $E_i\#_{\mathrm{fib}}E_S$ are smoothly isomorhic, here $\#_{\mathrm{fib}}$ denotes the fiberwise connected sum along the trivial sphere bundle $T^k \times S^3 \to T^k$.
This implies the assertion of the \lcnamecref{lem: stabilization}.
\end{proof}

\begin{lem}
\label{lem: 2-torsion}
We have $2\mathring{\alpha}_i=0$ and $2\alpha_i=0$.
\end{lem}

\begin{proof}
It follows from \cite[Lemma~3.6]{konno2022homological} that $(E_i)_\ast([T^k])$ is 2-torsion
 for every $i$.
 Thus $\mathring{\alpha_i}$ is also 2-torsion.
Since  $\alpha_i=\rho_\ast(\mathring{\alpha_i})$, we have $\alpha_i$ is 2-torsion as well. 
\end{proof}

The following is the most general statement of this paper:

\begin{thm}
\label{thm: generating homoogy}
Let $k>0$ and let $M$ be a smooth simply-connected closed oriented $4$-manifold that satisfies \cref{assumption on M}.
Set $X=M\# kS^2\times S^2$.
Then we have the following:
\begin{itemize}
\item [(i)] The set $\{\alpha_i \mid i\geq2 \}$ generates a direct summand of 
\[
\ker(i_\ast : H_k(B\mathrm{Diff}^+(X);\Z) \to H_k(B\Homeo^+(X);\Z))
\]
isomorphic to $(\Z/2)^\infty$.
\item [(ii)] The set $\{\mathring{\alpha}_i \mid i\geq2\}$ generates a direct summand of
\[
\ker(s_\ast : H_k(B\mathrm{Diff}_\del(\mathring{X});\Z) \to H_k(B\mathrm{Diff}_\del(\mathring{X}\# S^2 \times S^2);\Z))
\]
isomorphic to $(\Z/2)^\infty$.
\end{itemize}
\end{thm}

Due to \cref{lem: 2-torsion,lem: topologically trivial,lem: stabilization}, 
it suffices to prove that there is a homomorphism
\[
H_k(B\mathrm{Diff}_\del(\mathring{X});\Z) \to (\Z/2)^{\infty}
\]
that restricts to a surjection 
\[
\left<\mathring{\alpha}_i \mid i\geq2\right> \twoheadrightarrow (\Z/2)^{\infty}.
\]
We shall prove this remaining part in the subsequent sections.

Assuming \cref{thm: generating homoogy}, we obtain the proofs of results exhibited in \cref{section Main results}:

\begin{proof}[Proofs of \cref{thm: main general,thm: main stabilization}]
These follow from \cref{lem En satisfies Assumption,thm: generating homoogy}.   
(Note that $\mathrm{Diff}^+(X)=\mathrm{Diff}(X)$ and $\Homeo^+(X)=\Homeo(X)$, since $X$ in the assertions of \cref{thm: main general,thm: main stabilization} have non-zero signature.)
\end{proof}

\section{Characteristic classes from Seiberg--Witten theory}
\label{section Characteristic classes from Seiberg--Witten theory}

\subsection{Output}
\label{subsec: output of char class}

The proof of \cref{thm: generating homoogy} uses characteristic classes defined by using Seiberg--Witten theory.
Fix $k\geq0$, and let $X$ be a smooth closed oriented $4$-manifold with $b^+(X) \geq k+2$.
Lin and the author \cite{konno2022homological} defined a characteristic class
\begin{align}
\label{eq: char KLin}
\SWbbhalftot^k(X) \in H^k(B\mathrm{Diff}^+(X);\Z/2),  
\end{align}
which we called the half-total Seiberg--Witten characteristic class.
This was inspired by Ruberman's total Seiberg--Witten invariant of diffeomorphisms \cite{Rub01}, together with a gauge-theoretic construction of characteristic classes by the author \cite{K21}.
We introduce generalizations of the characteristic class \eqref{eq: char KLin} to prove \cref{thm: generating homoogy} as follows.

Let $\Spinc(X,k)$ denote the set of isomorphism classes of spin$^c$ structures $\fraks$ with $d(\fraks)=-k$, where $d(\fraks)$ is the formal dimension of the Seiberg--Witten moduli space:
\begin{align}
\label{eq: formal dim}
d(\fraks) = \frac{1}{4}(c_1(\fraks)^2-2\chi(X)-3\sigma(X)).    
\end{align}
The group $\Z/2$ acts on  $\Spinc(X,k)$ by the charge conjugation, which flips the sign of the first Chern class of a spin$^c$ structure.
Let $\Spinc(X,k)/\Conj$ denote the quotient of $\Spinc(X,k)$ under this $\Z/2$-action.
On the other hand, $\mathrm{Diff}^+(X)$ acts on $\Spinc(X,k)$ via pull-back.
Since the charge conjugation commutes with and the action of $\mathrm{Diff}^+(X)$ on $\Spinc(X,k)$, we have an action of $\mathrm{Diff}^+(X)$ on $\Spinc(X,k)/\Conj$.

Let $\calS$ be a subset of $\Spinc(X,k)/\Conj$ which is setwise preserved under the action of $\mathrm{Diff}^+(X)$.
We suppose that $\calS$ does not contain the coset of a self-conjugate spin$^{c}$ structure, which is needed to ensure a families perturbation can be taken to be non-zero and transverse.
We shall define a cohomology class
\begin{align}
\label{eq: SW halftot}
\SWbbhalftot^k(X,\calS)
\in H^k(B\mathrm{Diff}^+(X);\Z/2),    
\end{align}
by repeating the construction in \cite{konno2022homological} only for spin$^c$ structures $\fraks$ whose cosets $[\fraks]$ under the $\Z/2$-action lie in $\calS$.

Before explaning the construction of $\SWbbhalftot^k(X,\calS)$, it is worth looking at the lowest degree case to see which spin$^c$ structures involve: suppose $k=0$, and take a section $\tau : \Spinc(X,0)/\Conj \to \Spinc(X,k)$ of the quotient map $\Spinc(X,0) \to \Spinc(X,0)/\Conj$.
Then $\SWbbhalftot^0(X,\calS)$ is given by 
\[
\SWbbhalftot^0(X,\calS)
= \sum_{\fraks \in \tau(\calS)}SW(X,\fraks) \in \Z/2.
\]
Note that this number in $\Z/2$ is independent of $\tau$, determined only by $\calS$.

The characteristic class \eqref{eq: SW halftot} is a generalization of the characteristic class \eqref{eq: char KLin} given in \cite{konno2022homological}: by setting $\calS=\Spinc(X,k)/\Conj$, we obtain \eqref{eq: char KLin}, namely
\[
\SWbbhalftot^k(X,\Spinc(X,k)/\Conj)
= \SWbbhalftot^k(X).
\]

\subsection{Construction of the characteristic classes}
\label{subsection Construction of the characteristic classes}

We explain the construction of $\SWbbhalftot^k(X,\calS)$ below.
We omit some details which are completely analogous to arguments in  \cite{konno2022homological}:
see \cite[Section~2]{konno2022homological} for the full treatment.
First, let us recall the basics of the Seiberg--Witten equations.
To write down the (perturbed) Seiberg--Witten equations, we need to fix a spin$^c$ structure $\fraks$ on $X$, a Riemannian metric $g$ on $X$, and an imaginary-valed self-dual 2-form $\mu \in i\Omega^+_g(X)$ on $X$.
Here $\Omega^+_g(X)$ denotes the set of self-dual 2-forms for the metric $g$. 
The Seiberg--Witten equations perturbed by $\mu$ are of the form
\begin{align}
\label{eq: SW eq}
\left\{
\begin{array}{ll}
F_A^+=\sigma(\Phi, \Phi)+\mu,\\
D_A\Phi=0.
\end{array}
\right.
\end{align}
Here $A$ is a $U(1)$-connection of the determinant line bundle for $\fraks$, $\Phi$ is a positive spinor for $\fraks$, $\sigma(-,-)$ is a certain quadratic form, and $D_A$ is the spin$^c$ Dirac operator associated with $A$.
The Seiberg--Witten equations is $\Map(X,U(1))$-equivariant, and we define the moduli space of solutions to the Seiberg--Witten equations by
\[
\calM(X,\fraks,g,\mu)
= \{(A,\Phi) \mid \text{$(A,\Phi)$ satisfies }\eqref{eq: SW eq}\}/\Map(X,U(1)).
\]

Next, let us recall the charge conjugation symmetry on the Seiberg--Witten equations.
Let $\bar{\fraks}$ denote the conjugate spin$^c$ structure to $\fraks$, which satisfies $c_{1}(\bar{\fraks})=-c_{1}(\fraks)$.
Then there is a bijection
\begin{align}
\label{eq: conj symm}
c : \calM(X,\fraks,g,\mu)
\to \calM(X,\bar{\fraks},g,-\mu)
\end{align}
called the charge conjugation symmetry, which becomes a diffeomorphism between the moduli spaces if the perturbation $\mu$ is generic so that $\calM(X,\fraks,g,\mu)$ is a smooth manifold (then so is $\calM(X,\bar{\fraks},g,-\mu)$ automatically).

Let $\scrR(X)$ denote the space of Riemannian metrics.
Set 
\[
\Pi(X) = \bigcup_{g \in \scrR(X)}i\Omega^+_g(X).
\]
We think of $\Pi(X)$ as a vector bundle over the Frechet manifold $\scrR(X)$, and then take a fiberwise completion with respect to a suitable Sobolev norm.
Let us use the same notation $\Pi(X) \to \scrR(X)$ also for the Hilbert bundle obtained by this completion.
Let $\circPi(X)$ be the subset of $\Pi(X)$ consisting of perturbations $\mu$ such that:
\begin{itemize}
    \item $\|\mu\|\leq1$ for the Sobolev norm on $\Omega^+_g(X)$, and
    \item there is no reducible solution for $\mu$.
\end{itemize}
 The space $\circPi(X)$ is $(b^+(X)-2)$
-connected, and $\circPi(X)$ is invariant under the fiberwise $(-1)$-multiplication on the Hilbert bundle $\Pi(X) \to \scrR(X)$.

What makes the construction of the half-total Seiberg--Witten characteristic class complicated is the fact that the charge conjugation acts on the space of perturbations non-trivially; the action is given as (fiberwise) multiplication by $-1$.
Because of this, to implement a construction equivariantly under the charge conjugation, we need a `multi-valued perturbation' when we form a collection of moduli spaces over a set of spin$^c$ structures, not just a copy of a common families self-dual 2-form.
This is formulated as follows.

Let $\varpi : \Spinc(X,k) \to \Spinc(X,k)/\Conj$ be the quotient map.
Define $\tilde{\calS} := \varpi^{-1}(\calS) \subset \Spinc(X,k)$.
Since $\calS$ is invariant under the $\mathrm{Diff}^+(X)$-action, $\tilde{\calS}$ is also $\mathrm{Diff}^+(X)$-invariant.
Define 
\[
\circPi(X,\calS)' := (\tilde{\calS} \times \circPi(X))/(\Z/2),
\]
where $\Z/2$ acts on $\tilde{\calS}$ via the charge conjugation and on $\circPi(X)$ via the (fiberwise) $(-1)$-multiplication.
(To make our notation consistent with that in \cite{konno2022homological}, let us use the notation $\circPi(X,\calS)'$ with prime, not like $\circPi(X,\calS)$. This remark applies throughout this \lcnamecref{section Characteristic classes from Seiberg--Witten theory}.)

Now we consider a family of $4$-manifolds.
Let $X \to E \to B$ be a fiber bundle with structure group $\mathrm{Diff}^+(X)$ over a CW complex $B$.
For $b \in B$, we denote by $E_b$ the fiber of $E$ over $b$.
Associated with $E$, we have several natural fiber bundles.
For instance, since $\mathrm{Diff}^+(X)$ acts on $\calS$ via pull-back, we obtain an associated fiber bundle over $B$ with fiber $\calS$.
We denote it by
\[
\calS
\to \calS(E)
\to B.
\]
Similarly, we get a fiber bundle with fiber $\circPi(X,\calS)'$, denoted by 
\[
\circPi(X,\calS)'
\to \circPi(E,\calS)'
\to B.
\]
This has underlying families of spaces of metrics, denoted by 
\[
\scrR(X) \to \scrR(E) \to B.
\]
A section of $\scrR(E) \to B$ is a fiberwise metric on $E$.
Note that the forgetful map $\circPi(X,\calS)' \to \calS$ induces a surjection
\[
\circPi(E,\calS)' \to \calS(E),
\]
which commutes with the projections onto $B$.

It could be worth unpackaging the data $\circPi(E,\calS)'$.
Let $\vec{\mu}$ be a point in $\circPi(E,\calS)'$.
Let $b \in B$ and $g \in \scrR(E_b)$ be the images of $\vec{\mu}$ under the projections $\circPi(E,\calS)' \to B$ and $\circPi(E,\calS)' \to \scrR(E)$.
Let $\calS(E)_b$ be the fiber of $\calS(E) \to B$ over $b$.
Picking a representative $\fraks$ of each coset  $[\fraks] \in \calS(E)_b$, we can express
$\vec{\mu}$ as a collection of a self-dual 2-forms $\{\mu_{\fraks} \in \Omega^+_{g}(E_b)\}_{[\fraks] \in \calS(E)_b}$.
We set
\[
\calM(E_b, \calS, \vec{\mu})
= \bigsqcup_{[\fraks] \in \calS(E)_b} \calM(E_b, \fraks, g, \mu_{\fraks}).
\]
If all $\mu_{[\fraks]}$ are generic,  $\calM(E_b, \calS, \vec{\mu})$ is a smooth manifold. 
Further, as an unoriented manifold, $\calM(E_b, \calS, \vec{\mu})$ is independent of choice of representatives $\fraks$ of $[\fraks]$.
Indeed, if we choose the other representative $\bar{\fraks}$ of $[\fraks]$, the chosen perturbation becomes $\mu_{\bar{\fraks}} = -\mu_{\fraks}$, so we can use the diffeomorphism \eqref{eq: conj symm}.

Now let us take a fiberwise metric $\tilde{g} : B \to \scrR(E)$, and pick a section $\sigma' : \calS \to \circPi(E,\calS)'$ that makes the following diagram commutative:
\begin{align*}
 \begin{CD}
     \calS(E) @>{\sigma'}>> \circPi(E,\calS)' \\
  @VVV    @VVV \\
     B   @>{\tilde{g}}>>  \scrR(E).
  \end{CD}    
\end{align*}
Define the {\it half-total moduli space for $\sigma'$} by
\[
\calM_{\mathrm{half}}(E, \calS, \sigma')
= \bigcup_{b \in B} \calM(E_b, \calS, \sigma'(b)).
\]
(This was denoted by $\calM_{\sigma',\mathrm{half}}$ in \cite[Definition 2.11]{konno2022homological}, but let us use the notation $\calM_{\mathrm{half}}(E,\calS, \sigma')$ to keep track of $E$ and $\calS$.)
If $B$ is a comapct manifold, by choosing generic $\sigma'$,  $\calM_{\mathrm{half}}(E, \calS, \sigma')$ becomes a compact manifold too ({\it cf.}~\cref{lem: finiteness}), and
the dimension of $\calM_{\mathrm{half}}(E, \calS, \sigma')$ is given by $\dim{B}-k$.
In particular, for $\dim{B}=k$, we can define a $\Z/2$-valued invariant by counting the zero dimensional compact manifold $\calM_{\mathrm{half}}(E, \calS, \sigma')$.

For a general case where $B$ is neither compact nor a manifold, we define a cochain
\[
\SWcalhalftot^k(E, \calS, \sigma') \in C^k(B)
\]
as follows, where $C^k(B)$ denotes the $\Z/2$-coefficient cellular cochain group.
Loosely speaking, for each $k$-cell $e$ of $B$ with a characteristic map $\varphi_e : D^k \to B$, we define
\[
\SWcalhalftot^k(E, \calS, \sigma')(e)
= \#\calM_{\mathrm{half}}(\varphi_e^\ast E, \calS, \varphi_e^\ast \sigma') \in \Z/2.
\]
Here the right-hand side is a finite sum ({\it cf.}~\cref{lem: finiteness}), and we can justify the necessary transversality by using a virtual neighborhood technique, just as in \cite{K21,konno2022homological}.
Using the assumption that $b^+(X) \geq k+2$, we can prove that $\SWcalhalftot^k(E, \calS, \sigma')$ is a cocycle, and that the cohomology class
\[
\SWbbhalftot^k(E, \calS)
:= [\SWcalhalftot^k(E, \calS, \sigma')] 
\in H^k(B;\Z/2)
\]
is independent of the choice of $\sigma'$ (\cite[Propositions 2.22, 2.23]{konno2022homological}).
We set 
\[
\SWbbhalftot^k(X, \calS)
:=
\SWbbhalftot^k(E\mathrm{Diff}^+(X), \calS)
\in H^k(B\mathrm{Diff}^+(X);\Z/2).
\]

\subsection{Finiteness}

Here we record some finiteness result, which was used in \cref{subsection Construction of the characteristic classes} and is necessary in a subsequent argument too.

First, let us recall the following well-known finiteness of Seiberg--Witten moduli spaces (see, e.g., \cite{Mo96}). Fix a metric $g$ and $k \in \Z$.
Then there are only finitely many spin$^c$ structures $\fraks$ with $d(\fraks)=k$ for which the moduli space $\calM(X,\fraks,g,\mu)$ for the perturbed equations \eqref{eq: SW eq} are non-empty for some $\mu \in \Omega^+_g(X)$ with $\|\mu\| \leq 1$.
Here $\|-\|$ denotes a suitable Sobolev norm.
Moreover, for a fixed pair $(g,\mu)$, the moduli space $\calM(X,\fraks,g,\mu)$ is compact.
A families generalization of this fact in our context is as follows.

As in \cref{subsection Construction of the characteristic classes}, fix $k\geq0$, and let $X$ be a smooth closed oriented $4$-manifold with $b^+(X) \geq k+2$, and let $X \to E \to B$ be a fiber bundle with structure group $\mathrm{Diff}^+(X)$ over a CW complex $B$.

\begin{lem}
\label{lem: finiteness}
Suppose that $B$ is compact.
If we pick a section $\sigma'$ as in \cref{subsection Construction of the characteristic classes}, then the half-total moduli space 
\[
\calM_{\mathrm{half}}(E,\Spinc(X,k)/\Conj,\sigma')
\]
is compact.
\end{lem}

\begin{proof}
This follows from that we used perturbations with $\|\mu\| \leq 1$ in the definition of $\circPi(X)$.
\end{proof}

For our purpose, an important case is that $\calS$ is an orbit of the action of $\mathrm{Diff}^+(X)$ on $\Spinc(X,k)/\Conj$.
Set
\[
\Spinc(X,k)^{\vee} := \{ \fraks \in \Spinc(X,k) \mid \fraks \not\cong \bar{\fraks} \}
\]
and let $\mathbb{S}(X,k)$ denote the orbit space for the $\mathrm{Diff}^+(X)$-action on $\Spinc(X,k)^{\vee}/\Conj$,
\[
\mathbb{S}(X,k) = (\Spinc(X,k)^{\vee}/\Conj)/\mathrm{Diff}^+(X).
\]
As an analog of the notion of a basic class, we call $\calS \in \bbS(X,k)$ a {\it basic orbit of $E$} if $\SWbbhalftot^k(E,\calS) \neq 0$.
Let $\Bhalf(E,k)$ denote the set of basic orbits:
\[
\Bhalf(E,k)
= \{\calS \in \bbS(X,k) \mid \SWbbhalftot^k(E,\calS) \neq 0\}.
\]
Then we have:

\begin{lem}
\label{cor finiteness basic}
Suppose that $B$ is compact.
Then $\Bhalf(E,k)$ is a finite set. 
\end{lem}

\begin{proof}
Fix a section $\sigma'$.
\cref{lem: finiteness} implies that there are only finitely many $\fraks \in \Spinc(X,k)$ such that there is $b \in B$ with $\calM(E_b, \fraks, g_b, \sigma'(b)) \neq \emptyset$, where $g_b$ is the underlying metric of $\sigma'(b)$ on $E_b$.
Since $\#\Bhalf(E,k)$ is bounded above by the number of such $\fraks$, the assertion follows.   
\end{proof}

\section{Computing the invariant}
\label{section Computing the invariant}

In this \lcnamecref{section Computing the invariant},
we prove \cref{thm: generating homoogy} by evaluating the Seiberg--Witten characteristic classes $\SWbbhalftot^k(X,\calS)$ introduced in \cref{section Characteristic classes from Seiberg--Witten theory} at homology classes $\alpha_i$ defined in \eqref{eq: alpha}.

Precisely, we shall consider the homomorphism
\begin{align*}
\bigoplus_{\calS \in \bbS(X,k)}
\langle \SWbbhalftot^k(X,\calS), - \rangle :
H_k(B\mathrm{Diff}^+(X);\Z) \to \bigoplus_{\calS \in \bbS(X,k)} \Z/2.
\end{align*}
We shall show that this homomorphism has infinitely generated image in $\bigoplus_{\calS \in \bbS(X,k)} \Z/2$ for $4$-manifolds $X$ considered in \cref{thm: generating homoogy}.

\subsection{Reducing to the monodromy invariant part}

The Seiberg--Witten characteristic class $\SWbbhalftot^k(X, \calS)$ involves spin$^c$ structures that are not invariant under the monodromies of the families that we consider.
Adapting an argument in \cite[Section~3.1]{konno2022homological} to our setup, we shall see that such spin$^c$ structures do not contribute to the final computation.

To describe it, let us recall the numerical families Seiberg--Witten invariant.
Let $B$ be a closed smooth manifold of dimension $k \geq 0$, $X$ be a smooth oriented closed $4$-manifold of $b^+(X) \geq k+2$, and $X \to E \to B$ be a fiber bundle with structure group $\mathrm{Diff}^+(X)$ over $B$. 
Given a spin$^c$ structure $\fraks$ on $X$ of formal dimension $-k$,
suppose that the monodromy of $E$ fixes the isomorphism class of $\fraks$.
Then the numerical families Seiberg--Witten invariant 
\[
SW(E,\fraks) \in \Z/2
\]
can be defined.
If the structure of $E$ lifts to the automorphism group of the spin$^c$ $4$-manifold $(X,\fraks)$, this is the invariant defined by Li--Liu~\cite{LiLiu01}.
However, even if $E$ does not admit such a lift, one can still define $SW(E,\fraks)$ \cite{K21,BK20gluing}.

Pick an orbit $\calS \in \bbS(X,k)$.
We regard $\calS$ also as a subset of $\Spinc(X,k)/\Conj$.
Let $\tau : \Spinc(X,k)/\Conj \to \Spinc(X,k)$ be a section of the quotient map $\Spinc(X,k) \to \Spinc(X,k)/\Conj$.
For mutually commuting diffeomorphisms $f_1, \ldots, f_k$ of $X$, we denote by $X_{f_1, \ldots, f_k} \to T^k$ the multiple mapping torus of $f_1, \ldots, f_k$. 

\begin{pro}[{\it cf.} {\cite[Corollary~3.4]{konno2022homological}}]
\label{prop: reducing to monodromy invariant part}
Let $f_1, \dots, f_k : X \to X$ be mutually commuting orientation-preserving diffeomorphisms.
Suppose that they satisfy the following conditions:
\begin{itemize}
\item[(i)] For each $i = 1, \ldots, k$, $f_i$ preserves $\tau(\calS)$ setwise.
\item[(ii)] For each $i$, there exists a smooth isotopy $(F_i^t)_{t \in [0,1]}$ from $f_i^2$ to $\id_X$.
For $i\neq j$, $F_i^t$ commutes with $f_j$ for any $t \in [0,1]$.
\end{itemize}
Then we have
\[
\langle
\SWbbhalftot^k(X_{f_1, \dots, f_k}, \calS), [T^k]
\rangle
= \sum_{\substack{\fraks \in \tau(\calS), \\ f_i^\ast \fraks = \fraks\ (1 \leq i \leq k)}}
SW(X_{f_1, \dots, f_k},\fraks)
\]
in $\Z/2$.
\end{pro}

\begin{proof}
The proof is obtained by repeating the proof of \cite[Corollary~3.4]{konno2022homological} with replacing $\Spinc(X,k)/\Conj$ with $\calS$.
We just give a slight comment on how to do the modification.

If the actions of all $f_i$ on $\tau(\calS)$ are trivial, there is nothing to prove.
To treat the other case,
first note that we have a modification 
of \cite[Lemma~3.3]{konno2022homological} obtained by replacing $\Spinc(X,k)/\Conj$ with $\calS$.
Let us consider a $(\Z/2)^k$-action on $\tau(\calS)$ generated by $f_1, \dots, f_k$.
For $\fraks \in \tau(\calS)$, if there is $i$ with $f_i^\ast\fraks\neq\fraks$, we may use the modified \cite[Lemma~3.3]{konno2022homological} to conclude that the sum of the counts of the moduli spaces for the $(\Z/2)^k$-orbit of $\fraks$ is zero over $\Z/2$.
Thus, in any case, $\langle
\SWbbhalftot^k(X_{f_1, \dots, f_k}, \calS), [T^k]
\rangle$ is computed from the counts of the moduli spaces only for the monodromy invariant spin$^c$ structures, and it ends up with the assertion of \cref{prop: reducing to monodromy invariant part}.
\end{proof}

\subsection{Gluing result}
\label{gluing section}

Another thing we need is a gluing result proven by Baraglia and the author  \cite{BK20gluing}.
We recall the statement for readers' convenience.
In general, let $Z$ be an oriented smooth closed $4$-manifold, and $Z \to E \to B$ be an oriented smooth fiber bundle with fiber $Z$.
Then we get an associated vector bundle 
\[
\R^{b^+(Z)} \to H^+(E) \to B,
\]
by considering maximal-dimensional positive-definite subspaces of the second cohomology fiberwise.
The isomorphism class of $H^+(E)$ is determined only by $E$.

The gluing result we need is formulated as follows.
Let $k>0$, and let $M$, $N$ be closed oriented smooth $4$-manifolds with $b^+(M) \geq 2$ and $b^+(N)=k$, and with $b_1(M)=b_1(N)=0$.
Set $X=M\#N$.
Let $\frakt \in \Spinc(M,0)$ and $\frakt' \in \Spinc(N,k+1)$.
Then we have $d(\frakt\#\frakt')=-k$.
Let $B$ be a closed smooth manifold of dimension $k$, and $M \to E_M \to B$ and $N \to E_N \to B$ be oriented smooth fiber bundles.
Fix sections $\iota_M : B \to E_M$, $\iota_N : B \to E_N$ whose normal bundles are isomorphic via a fiberwise orientation-reversing isomorphism, so that we can form the fiberwise connected sum $X \to E_X \to B$ of $E_M$ and $E_N$ along $\iota_M, \iota_N$.
Then we have:

\begin{thm}[{\cite[Theorem 1.1]{BK20gluing}}]
\label{gluing BK}
If $w_{b^+(N)}(H^+(E_N)) \neq 0$, then we have
\[
SW(E_X,\frakt\#\frakt') = SW(M,\frakt)
\]
in $\Z/2$.
\end{thm}

Now we apply \cref{gluing BK} to the multiple mapping torus $E_i \to T^k$ constructed in \cref{subsectionFamilies over the torus} for $i \geq 1$.
For each $j=1, \dots, k$,
recall that $f_j$ acts on the $j$-th copy of $H^+(S^2 \times S^2) \subset H^2(S^2\times S^2)$ via multiplication by $-1$.
we can see that the vector bundle 
\[
H^+((kS^2\times S^2)_{f_1, \dots, f_k}) \to T^k
\]
associated to the multiple mapping torus $(kS^2\times S^2)_{f_1, \dots, f_k} \to T^k$ satisfies 
\begin{align}
\label{eq: Stiefel-Whitney}
w_{k}(H^+((kS^2\times S^2)_{f_1, \dots, f_k})) \neq 0.    
\end{align}
Let $\fraks_S$ denote the unique spin structure on $kS^2\times S^2$.
Then we have $\fraks_S \in \Spinc(kS^2\times S^2,k+1)$.

\begin{lem}
\label{lem: conseq gluing}
Let $\frakt \in \Spinc(M_i,0)$.
Then we have
\[
SW(E_i, \frakt \# \fraks_S)  = SW(M_i,\frakt)
\]
in $\Z/2$.    
\end{lem}
\begin{proof}
This follows from \eqref{eq: Stiefel-Whitney} and \cref{gluing BK}.  
\end{proof}

\subsection{Completion of the proof}

As in \cref{sec: Families over a torus}, fix $k>0$, take a $4$-manifold $M$ satisfying  \cref{assumption on M}. 
We shall use $M_i$ and $c_i$ that appear in \cref{assumption on M}, and shall use the notation $E_i$ and $\alpha_i$ for $i\geq1$ introduced in \cref{subsectionFamilies over the torus}.
Set $X = M\#kS^2 \times S^2$ and $X_i = M_i\#kS^2\times S^2$.

For each $i\geq1$, we fix a section
\[
\tau^0_i : \Spinc(M_i)/\Conj \to \Spinc(M_i)
\]
of the quotient map $\Spinc(M_i) \to \Spinc(M_i)/\Conj$.
Using $\tau^0_i$, we define a section 
\[
\tau_i : \Spinc(X_i)/\Conj \to \Spinc(X_i)
\]
as follows: for $\fraks \in \Spinc(X_i)$, we define $\tau([\fraks])$ to be the spin$^c$ structure $\fraks'$ with $[\fraks] = [\fraks']$ in $\Spinc(X_i)/\Conj$ such that $\fraks'|_{M_i} = \tau_0([\fraks|_{M_i}])$.
Restricting this, we obtain a section (denoted by the same notation)
\[
\tau_i : \Spinc(X_i,k)/\Conj \to \Spinc(X_i,k).
\]

As in \cref{gluing section}, let $\fraks_S$ denote the unique spin structure on $kS^2\times S^2$.
For each $i \geq 1$, we define $\calS_i \in \bbS(X_i,k)$ to be the $\mathrm{Diff}^+(X_i)$-orbit that contains $[c_i\#\fraks_S] \in \Spinc(X_i,k)/\Conj$.

\begin{pro}
\label{cor reduction to monodromy inv part}
For $E_i \to T^k$ constructed in \cref{subsectionFamilies over the torus}, we have
\[
\langle
\SWbbhalftot^k(X, \calS_i), (E_i)_\ast([T^k])
\rangle
\neq 0
\]
in $\Z/2$.
\end{pro}

\begin{proof}
First, the naturality of the characteristic class implies that
\begin{align}
\label{eq: natural}
\langle
\SWbbhalftot^k(X, \calS_i), (E_i)_\ast([T^k])
\rangle
=
\langle
\SWbbhalftot^k(E_i, \calS_i), [T^k]
\rangle.
\end{align}
To compute the right-hand side of \eqref{eq: natural}, we shall apply \cref{prop: reducing to monodromy invariant part} to the families $E_i$.  
Recall that $E_i$ was constructed by using a diffeomorphism $f \in \mathrm{Diff}_\del(S^2 \times S^2 \setminus \mathrm{Int}(D^4))$.
This diffemorphism $f$ is order 2 in $\pi_0(\mathrm{Diff}_\del(S^2 \times S^2 \setminus \mathrm{Int}(D^4)))$.
Thus, for the diffeomorphisms $f_1, \dots, f_k$ on $X_i$, of which the multiples mapping torus is $E_i$, we can find isotopies $(F_i^t)_{t \in [0,1]}$ that satisfies the assumption (ii) of \cref{prop: reducing to monodromy invariant part}.
Since $f_j$ act trivially on $M_i$, by the construction of $\tau_i$, it follows that $\tau_i(\calS_i)$ is setwise preserved under the actions of $f_j$.
Thus we may apply \cref{prop: reducing to monodromy invariant part} to the families $E_i$ and obtain the equality
\begin{align}
\label{eq: first eq}
\langle
\SWbbhalftot^k(E_i, \calS_i), [T^k]
\rangle
= \sum_{\substack{\fraks \in \tau_i(\calS_i), \\ f_j^\ast \fraks = \fraks\ (1 \leq j \leq k)}}
SW(E_i,\fraks)
\end{align}
in $\Z/2$.

We shall compute the right-hand side of \eqref{eq: first eq}.
Since $f_j$ acts on the $j$-th copy of $H^2(S^2\times S^2)$ via multiplication by $-1$, a spin$^c$ structure $\fraks \in \Spinc(X_i)$ is $f_j$-invariant for all $j$ if and only if $\fraks$ is of the form $\frakt \# \fraks_S$, where $\frakt \in \Spinc(M_i)$.
It is easy to see that, if $d(\frakt \# \fraks_S)=-k$, then $d(\frakt)=0$.
Thus we get from \cref{lem: conseq gluing} that 
\begin{align}
\label{eq: second eq}
\sum_{\substack{\fraks \in \tau_i(\calS_i), \\ f_j^\ast \fraks = \fraks\ (1 \leq j \leq k)}}
SW(E_i,\fraks)
= \sum_{\substack{\frakt\#\fraks_S \in \tau_i(\calS_i), \\ \frakt \in \Spinc(M_i,0)}}
SW(M_i,\frakt)
\end{align}
in $\Z/2$.

To compute the right-hand side of \eqref{eq: second eq},
let $\frakt \in \Spinc(M_i,0)$ be a spin$^c$ structure on $M_i$.
We claim that  $\frakt\#\fraks_S$ lies in $ \tau_i(\calS_i)$ if and only if all of the following three conditions (i)-(iii) are satisfied: (i) $\divi(c_1(\frakt)) = \divi(c_i)$, (ii) $c_1(\frakt)^2 = c_i^2$, and (iii) $\frakt \in \tau_i^0(\Spinc(M_i,0)/\Conj)$.
Noting $c_1(\frakt) = c_1(\frakt \# \fraks_S)$ in $H^2(X_i;\Z)$, this claim is a direct consequence of \cref{lem: Wall}, which we shall see later.

By the claim of the last paragraph, we have 
\begin{align}
\label{eq: third eq}
\sum_{\substack{\frakt\#\fraks_S \in \tau_i(\calS_i), \\ \frakt \in \Spinc(M_i,0)}}
SW(M_i,\frakt)
= N(M_i; c_i)
\end{align}
in $\Z/2$.
Here the right-hand side of \eqref{eq: third eq} was assumed to be non-zero in $\Z/2$ in \cref{assumption on M}.
Thus the assertion of the \lcnamecref{cor reduction to monodromy inv part} follows from \eqref{eq: natural}, \eqref{eq: first eq}, \eqref{eq: second eq}, \eqref{eq: third eq}.
\end{proof}

Here we record a \lcnamecref{lem: Wall} that we have used above:

\begin{pro}[Wall~\cite{Wall62unimodular,Wa64}]
\label{lem: Wall}
Let $Z$ be a smooth closed oriented simply-connected $4$-manifold.
Suppose that $b_2(Z)-\sigma(Z) \geq 2$, and that $Z$ is either indefinite or $b_2(Z) \leq 8$.
Set $Z' = Z\#S^2\times S^2$.
Then, given $x, y \in H^2(Z';\Z)$, there exists $f \in \mathrm{Diff}^+(Z')$ with $f^\ast x=y$ if and only if $x, y$ have the same divisibility, self-intersection, and type (i.e. characteristic or not).
\end{pro}

\begin{proof}
For a unimodular lattice $Q$ with $\mathrm{rank}(Q)-\sigma(Q) \geq 4$, 
Wall \cite[page 337]{Wall62unimodular} proved that $\Aut(Q)$ acts transitively on elements of given divisibility, self-intersection, and type.
On the other hand, each of divisibility, self-intersection, and type is invariant under the action of $\Aut(Q)$.
Thus orbits in $Q/\Aut(Q)$ one-to-one correspond to triples consisting of divisibility, self-intersection, and type.
The assertion of the \lcnamecref{lem: Wall} follows from this applied to the intersection form of $Z'$, together with another theorem by Wall~\cite[Theorem~2]{Wa64} on the realizability of an automorphism of the intersection form by a diffeomorphism.
\end{proof}

Now we can complete the proof of the most general result in this paper:

\begin{proof}[Proof of \cref{thm: generating homoogy}]
As in the construction of $E_i$, we fix diffeomorphisms 
$\psi_i : \mathring{M_i}\#kS^2\times S^2 \to \mathring{X}$ and its extensions $\psi_i : M_i\#kS^2\times S^2 \to X$.
Considering the pull-back of the orbits $\calS_i \in \bbS(X_i,k)$ under $\psi_i$, we obtain orbits (denoted by the same notation) $\calS_i \in \bbS(X,k)$.

Passing to a subsequence if necessary, we may suppose that all $\divi(c_i)$ are distinct by (iii) of \cref{assumption on M}.
Thus we may suppose that all $\calS_i$ are distinct elements in $\bbS(X,k)$.
From this together with \cref{cor finiteness basic}, by passing to a subsequence again, we may suppose that
\begin{align}
\label{eq: disjoint basic classes}
\calS_i \notin \Bhalf(E_1,k) \cup \cdots \cup \Bhalf(E_{i-1},k)
\end{align}
for all $i \geq 2$.

Now it follows from \cref{cor reduction to monodromy inv part} together with \eqref{eq: alpha}, \eqref{eq: disjoint basic classes} that the homomorphism
\[
\bigoplus_{i\geq2}
\langle \SWbbhalftot^k(X,\calS_i), - \rangle :
H_k(B\mathrm{Diff}^+(X);\Z) \to \bigoplus_{i \geq 2} \Z/2
\]
restricts to a surjection
\[
\left<\mathring{\alpha}_i \mid i\geq2\right> \twoheadrightarrow  \bigoplus_{i \geq 2} \Z/2.
\]
This combined with \cref{lem: 2-torsion} implies that the subgroup $\left<\mathring{\alpha}_i \mid i\geq2\right>$ is a $(\Z/2)^{\infty}$-summand of $H_k(B\mathrm{Diff}^+(X);\Z)$, which together with \cref{lem: topologically trivial} completes the proof of (i) of \cref{thm: generating homoogy}.

Since $\rho_\ast(\mathring{\alpha}_i)=\alpha_i$, we obtain (ii) of \cref{thm: generating homoogy} from (i) of \cref{thm: generating homoogy} together with \cref{lem: stabilization}.
\end{proof}

\section{Addenda}

\subsection{Finiteness of mapping class groups in dimension $\neq $ 4}
\label{subsection Finiteness dim not 4}

In dimension $\neq$ 4, not only finite generation, but stronger finiteness on mapping class groups is known.

\subsubsection{dimension $\geq 6$}

Given a simply-connected closed smooth manifold $X$ of $\dim X\geq6$,
Sullivan~\cite[Theorem~(13.3)]{Sull77} proved that $\pi_0(\mathrm{Diff}(X))$ is ``commensurable" with an arithmetic group.
Krannich and Randal-Williams \cite{KrannichRW20} clarified that the term ``commensurable" is used in \cite{Sull77} in a different way from the current common usage.
In summary, given a group, we have implications:
\begin{align*}
&\text{(commensurable with an arithmetic group in the  current common sense)}\\
\Rightarrow& \text{(commensurable with an arithmetic group in the sense of \cite{Sull77})}\\
\Rightarrow & \text{(finitely presented)}
\Rightarrow  \text{(finitely generated)}.
\end{align*}
In particular, \cref{thm: main pi0} implies that mapping class groups of simply-connected $4$-manifolds need not be commensurable with arithmetic groups, even in Sullivan's sense.

\subsubsection{dimension $5$}

While the above result by Sullivan~\cite[Theorem~(13.3)]{Sull77} was stated in $\dim \geq 6$, actually his result holds also in dimension 5.
We record a way to deduce this from a recent paper \cite{bustamante2023finiteness}.
(The author thanks Sander Kupers for informing the author of this argument.)
In the proof of \cite[Theorem~2.6]{bustamante2023finiteness}, the assumption that $\dim\geq6$ was used only in the point (i) in the proof, but it follows from Cerf's theorem \cite{Cerf70} that $\pi_0(C^{\mathrm{Diff}}(M)) = 0$ for a simply-connected 5-manifold $M$, and the assumption that $\dim\geq6$ was not used in \cite[Proposition~2.7]{bustamante2023finiteness}, except for the part where \cite[Proposition~2.6]{bustamante2023finiteness} was used.

\subsubsection{dimension $\leq 3$}
The mapping class groups of closed orientable manifolds of dim $\leq 3$ are finitely presented. See Dehn~\cite{HatcherThurston80} for dimension 2. 
In dimension 3, a more general finiteness holds for the moduli space of 3-manifolds.
See Boyd--Bregman--Steinebrunner~\cite[Theorem 6.12]{boyd2024modulispaces3manifoldsboundary}.

\subsection{Questions: finiteness in other categories}

We close this paper by posting questions on categories other than the smooth category.

As noted in \cref{rem: top MCG}, for a simply-connected closed topological $4$-manifold $X$, the topological mapping class group $\pi_0(\Homeo(X))$ is known to be finitely generated, and so is $H_1(\BHomeo(X);\Z)$.

On the other hand, to the best of the author's knowledge, there is no known finiteness result on $H_k(\BHomeo(X))$ for $k>1$ for general simply-connected $4$-manifolds $X$.
However, it may be natural to hope such finiteness results in the $4$-dimensional topological category, as opposed to the smooth category:

\begin{ques}
Let $X$ be a simply-connected closed oriented topological $4$-manifold.
Is $H_k(\BHomeo(X);\Z)$ finitely generated for each $k$?   
\end{ques}

Recently, Lin and Xie \cite{lin2023configuration} extensively studied the {\it moduli space $\calM^{fs}(X)$ of formally smooth $4$-manifolds}, which is a middle moduli space between the smooth moduli space $\calM^{s}(X)=\BDiff(X)$ and the topological moduli space $\calM^{t}(X)=\BHomeo(X)$.
Lin and Xie pointed out that most exotic phenomena detected by gauge theory are relevant to the discrepancy between $\calM^{s}(X)$ and $\calM^{fs}(X)$.
Since infiniteness of $\calM^{s}(X)$ detected in this paper comes from gauge theory, it may be natural to expect finiteness of $\calM^{fs}(X)$:

\begin{ques}
Let $X$ be a simply-connected closed oriented topological $4$-manifold that admits a formally smooth structure.
Is $H_k(\calM^{fs}(X);\Z)$ finitely generated for each $k$?   
\end{ques}

\bibliographystyle{plain}
\bibliography{mainref}
\end{document}